\documentclass[12pt]{article}

\usepackage{amsmath, amsthm, amssymb,bm}

\numberwithin{equation}{section}

\newcommand{\mcal}[1]{\mathcal{#1}}

\DeclareMathOperator{\Tr}{Tr}

\def\diag{\mathrm{diag}}

\def\Wg{\mathrm{Wg}}

\def\({ \left( }
\def\){ \right)}

\def\trans#1{\mathord{#1}^{\mathrm{t}}}

\theoremstyle{plain}
\newtheorem{thm}{Theorem}[section]
\newtheorem{prop}[thm]{Proposition}
\newtheorem{lem}[thm]{Lemma}

\theoremstyle{definition}

\newtheorem{example}{Example}[section]
\newtheorem{remark}{Remark}[section]
\theoremstyle{conjecture}

\theoremstyle{problem}

\title{\textbf{Integration of invariant matrices and application to statistics}}
\author{\textsc{Beno\^{i}t Collins, Sho Matsumoto, and Nadia Saad}}

\date{\empty}

\begin{document}

\maketitle

\begin{abstract}
We consider random matrices that have invariance properties under the action of unitary groups (either a left-right invariance, or a conjugacy invariance), and we give formulas for moments in terms of functions of eigenvalues.
Our main tool is the Weingarten calculus.
As an application to statistics, we obtain new formulas for the pseudo inverse of Gaussian matrices and for the inverse of compound Wishart matrices.
\end{abstract}

\tableofcontents

\section{Introduction}

Wishart matrices have been introduced and studied for the first time for statistical purposes in \cite{Wishart},
and they are still a fundamental random matrix model related to theoretical statistics.
One generalization of Wishart matrices
is compound Wishart matrices which are studied, for example, in \cite{speicher,hiai-petz}. Compound Wishart matrices
appear in many topics such as the in the context of spiked random matrices.

The study of eigenvalues of Wishart matrices is quite well developed but a systematic of the joint moments
 of their entries (that we will call \emph{local moments}) is more recent.
On the other hand, the theoretical study of the inverse of Wishart matrices is also very important, in particular
for mathematical finance purposes, as shown in (\cite{CarvWest} and \cite{CarvMasWest})

However, the study of their local moments is much more recent, and was actually still open in the case
of the inverse of the compound Wishart matrix.

The aim of this paper is to provide a unified approach to the problem of computing the local moments
of the above random matrix models. We actually tackle a much more general setup, where we consider
any random matrix provided that its distribution has an appropriate invariance property (orthogonal or unitary)
 under an appropriate action (by conjugation, or by a left-right action).

Our approach is based on Weingarten calculus.
This tool is used to compute the local moments of random matrices
distributed according to the Haar measure on compact groups such as the unitary or the orthogonal group.

It was introduced in \cite{We78} and then improved many times, with a first complete description in \cite{Co03,CS06}.

In this paper, we need to introduce a modification of the Weingarten function, namely, a `double' Weingarten function
with two dimension parameters instead of one. As far as we know it is the first time that such a double-parameter
Weingarten function is needed. Beyond proving to be efficient in computing systematically moments, we believe that
it will turn out to have important theoretical properties.

As an interesting byproduct of our study - and as a preliminary to the solution of our problem of computing
the moments of the inverse of a compound Wishart random matrix, we obtain explicit moment formulas for
the pseudo-inverse of Ginibre random matrices.

This paper is organized as follows. In section \ref{sec:preliminary}, we recall known results about the moments of Wishart matrices,
their inverses, and about Weingarten calculus. Section \ref{sec:Invariance} is devoted to the computation of moments of
general invariant random matrices, and in section \ref{sec:application}, we solve systematically the problem of computing moments
of inverses of compound Wishart matrices.

\section{Preliminary}\label{sec:preliminary}

\subsection{Notation}

\subsubsection{Notation in the complex case} \label{sec:NotationC}

Let $k$ be a positive integer.
A partition of $k$ is a weakly decreasing sequence $\lambda=(\lambda_1,\dots,\lambda_l)$
of positive integers with $\sum_{i=1}^l \lambda_i=k$.
We write the length $l$ of $\lambda$ as $\ell(\lambda)$.

Let $S_k$ be the symmetric group acting on $[k]=\{1,2,\dots,k\}$.
A permutation $\pi \in S_k$ is decomposed into cycles.
If the numbers of lengths of cycles are $\mu_1 \ge \mu_2 \ge \dots \ge \mu_l$,
then the sequence $\mu=(\mu_1,\mu_2,\dots,\mu_l)$ is a partition of $k$.
We will refer to $\mu$ as the cycle-type of $\pi$.
Denote by $\kappa(\pi)$ the length $\ell (\mu)$ of the cycle-type of $\pi$, or equivalently
the number of cycles of $\pi$.

For two sequences $\bm{i}=(i_1,\dots,i_k)$ and $\bm{i}'=(i_1',\dots,i_k')$ of positive
integers and for a permutation $\pi \in S_k$, we put
$$
\delta_{\pi}(\bm{i},\bm{i}')= \prod_{s=1}^k \delta_{i_{\pi(s)}, i_{s}'}.
$$
Given a square matrix $A$ and a permutation
$\pi \in S_k$ of cycle-type $\mu=(\mu_1,\dots,\mu_l)$,
we define
$$
\Tr_\pi(A)= \prod_{j=1}^l \Tr (A^{\mu_j}).
$$

\begin{example} \label{example:unitary_notation}
Let
$\pi= \(\begin{smallmatrix} 1 & 2 & 3 & 4 & 5 & 6 & 7 & 8 \\
2 & 5 & 4 &3 &1 & 8 & 7 & 6 \end{smallmatrix}\) \in S_8$.
Then $\pi$ is decomposed as $\pi=(1 \ 2 \ 5)(3 \ 4)(6 \ 8)(7)$
and the cycle-type of $\pi$ is the partition $(3,2,2,1)$, which gives
$\kappa(\pi)=4$ and
$\Tr_\pi(A)=\Tr (A^3) \Tr(A^2)^2 \Tr(A)$.
\end{example}

\subsubsection{Notation in the real case}  \label{sec:NotationR}

Given $\sigma \in S_{2k}$, we attach an undirected graph $\Gamma(\sigma)$ with vertices
$1,2,\dots, 2k$ and edge set consisting of
$$
\big\{ \{2i-1,2i\} \ | \ i =1,2,\dots,k \big\} \cup
\big\{ \{\sigma(2i-1),\sigma(2i)\} \ | \ i =1,2,\dots,k \big\}.
$$
Here we distinguish every edge $\{2i-1,2i\}$ from $\{\sigma(2j-1),\sigma(2j)\}$
even if these pairs coincide.
Then each vertex of the graph lies on exactly two edges, and
the number of vertices in each connected component is even.
If the numbers of vertices are
$2\mu_1 \ge 2\mu_2 \ge \dots \ge 2\mu_l$ in the components,
then the sequence $\mu=(\mu_1,\mu_2,\dots,\mu_l)$ is a partition of $k$.
We will refer to the $\mu$ as the coset-type of $\sigma$, see \cite[VII.2]{Macdonald}
for detail.
Denote by $\kappa'(\sigma)$ the length $l(\mu)$ of the coset-type of $\sigma$, or equivalently
the number of components of $\Gamma(\sigma)$.

Let $M_{2k}$ be the set of all pair partitions of the set $[2k]=\{1,\dots,2k\}$.
A pair partition $\sigma \in M_{2k}$ can be uniquely expressed in the form
$$
\sigma= \big\{ \{ \sigma(1), \sigma(2)\},\{ \sigma(3),\sigma(4)\}, \dots,
\{\sigma(2k-1), \sigma(2k)\} \big\}
$$
with $1=\sigma(1)<\sigma(3)<\cdots< \sigma(2k-1)$ and $\sigma(2i-1)<\sigma(2i)$ $(1 \le i \le k)$.
Then $\sigma$ can be regarded as a permutation
$\( \begin{smallmatrix} 1 & 2 & \dots & 2k \\ \sigma(1) & \sigma(2) & \dots & \sigma(2k) \end{smallmatrix}\)$ in $S_{2k}$.
We thus embed $M_{2k}$ into $S_{2k}$.
In particular, the coset-type and the value of $\kappa'$ for $\sigma \in M_{2k}$ are defined.

For a permutation $\sigma \in S_{2k}$ and
a $2k$-tuple $\bm{i}=(i_1, i_2,\dots, i_{2k})$ of positive integers,
we define
$$
\delta'_{\sigma}(\bm{i})= \prod_{s=1}^{k} \delta_{i_{\sigma(2s-1)},i_{\sigma(2s)}}.
$$
In particular, if $\sigma \in M_{2k}$, then $\delta'_{\sigma}(\bm{i})= \prod_{\{a,b\} \in \sigma}
\delta_{i_a,i_b}$,
where the product runs over all pairs in $\sigma$.
For a square matrix $A$ and $\sigma \in S_{2k}$ with coset-type $(\mu_1, \mu_2, \dots, \mu_l)$, we define
$$
\Tr_{\sigma}'(A)=\prod_{j=1}^l \Tr (A^{\mu_j}).
$$

\begin{example}
Let
$\pi= \(\begin{smallmatrix} 1 & 2 & 3 & 4 & 5 & 6 & 7 & 8 \\
2 & 5 & 4 &3 &1 & 8 & 7 & 6 \end{smallmatrix}\) \in S_8$
as in Example \ref{example:unitary_notation}.
Then the coset-type of $\pi$ is the partition $(3,1)$, which gives
$\kappa'(\pi)=2$ and
$\Tr_\pi'(A)=\Tr (A^3)\Tr(A)$.
\end{example}

\subsection{Weingarten calculus}\label{sec:Wg}

\subsubsection{Unitary Weingarten calculus}

We review some basic material on unitary integration and unitary Weingarten function. A more complete exposition of these matters can be found in \cite{Co03, CS06, MN}.
We use notation defined in \S \ref{sec:NotationC}.

Let $L(S_{k})$ be the algebra of complex-valued functions on $S_k$ with  convolution
$$
(f_1 *f_2) (\pi) = \sum_{ \tau \in S_k} f_1(\tau) f_2(\tau^{-1} \pi) \qquad (
f_1,f_2 \in L(S_k), \ \pi \in S_k).
$$
The identity element in the algebra $L(S_k)$ is the Dirac function $\delta_e$
at the identity permutation $e=e_k \in S_k$.

Let $z$ be a complex number and consider the function $z^{\kappa(\cdot)}$ in $L(S_k)$ defined by
$$
S_k \ni \pi \mapsto z^{\kappa(\pi)} \in \mathbb{C},
$$
which belongs to the center $\mcal{Z}(L(S_k))$ of $L(S_k)$.
The \emph{unitary Weingarten function}
$$
S_k \ni \pi \mapsto \Wg^{\mathrm{U}} (\pi;z) \in \mathbb{C}
$$
is, by definition, the pseudo-inverse element of $z^{\kappa(\cdot)}$ in $\mcal{Z}(L(S_k))$,
i.e.,
the unique element in $\mcal{Z}(L(S_k))$ satisfying
$$
z^{\kappa(\cdot)} * \Wg^{\mathrm{U}}(\cdot;z) * z^{\kappa(\cdot)} = z^{\kappa(\cdot)}
\qquad \text{and} \qquad \Wg^{\mathrm{U}}(\cdot ;z) *z^{\kappa(\cdot)} * \Wg^{\mathrm{U}}(\cdot ;z)=
\Wg^{\mathrm{U}}(\cdot ;z).
$$

The expansion of the unitary Weingarten function in terms of irreducible characters $\chi^\lambda$
of $S_k$ is given by
$$
\Wg^{\mathrm{U}} (\pi;z) = \frac{1}{k!}
\sum_{\begin{subarray}{c} \lambda \vdash k \\ C_{\lambda}(z) \not= 0 \end{subarray}}
\frac{f^\lambda}{C_\lambda(z)} \chi^\lambda(\pi) \qquad (\pi \in S_{k}),
$$
summed over all parititons $\lambda$ of $k$ satisfying $C_\lambda(z) \not =0$.
Here $f^\lambda=\chi^\lambda(e)$ and
$$
C_\lambda(z)= \prod_{i=1}^{\ell(\lambda)} \prod_{j=1}^{\lambda_i} (z+j-i).
$$
In particular, unless $z \in \{0, \pm 1, \pm 2, \dots, \pm (k-1)\}$,
functions $z^{\kappa(\cdot)}$ and $\Wg^{\mathrm{U}}(\cdot ;z)$ are inverse of each other and
satisfy $z^{\kappa(\cdot)} * \Wg^{\mathrm{U}} (\cdot;z) = \delta_e$.

\begin{prop}[\cite{Co03}]
\label{prop:UnitaryWeingartenCalculus}
Let $U = (u_{ij})_{1\leq i,j\leq n}$ be an $n \times n$ Haar-distributed unitary matrix. For four sequences
$\bm{i} = (i_1, i_2, \dots , i_{k})$, $\bm{j} = (j_1, j_2,\dots, j_{k})$,
$\bm{i}' = (i_1', i_2', \dots , i_{k}')$, $\bm{j}' = (j_1', j_2',\dots, j_{k}')$
of positive integers in $[n]$,
we have
\begin{equation}
E[ u_{i_1j_1} \dots u_{i_{k}j_{k}} \overline{u_{i'_{1}j'_{1}} \cdots u_{i'_{k}j'_{k}}}] = \sum_{\sigma, \tau \in S_k}
\delta_{\sigma}(\bm{i},\bm{i}')
\delta_{\tau}(\bm{j},\bm{j}') \Wg^{\mathrm{U}}(\sigma^{-1}\tau;n).
\end{equation}
\end{prop}

We will need the following function later.
Define the function $\Wg^{\mathrm{U}}(\cdot;z,w)$ on $S_k$
with two complex parameters $z,w \in \mathbb{C}$ by the convolution
\begin{equation}  \label{eq:double-unitaryWg}
\Wg^{\mathrm{U}}(\cdot ;z,w) = \Wg^{\mathrm{U}}(\cdot;z) * \Wg^{\mathrm{U}}(\cdot;w).
\end{equation}
More precisely,
$$
\Wg^{\mathrm{U}}(\cdot ;z,w) =
\frac{1}{k!}
\sum_{\begin{subarray}{c} \lambda \vdash k \\ C_{\lambda}(z) C_{\lambda}(w) \not= 0 \end{subarray}}
\frac{f^\lambda}{C_\lambda(z)C_\lambda(w)} \chi^\lambda.
$$

\subsubsection{Orthogonal Weingarten calculus}

We next review theory  on orthogonal integration and orthogonal Weingarten function.
See \cite{CS06,CM09, M_ortho, M_Wishart, M_COE} for detail.
We use notation defined in \S \ref{sec:NotationR}.

Let $H_k$ be the hyperoctahedral group of order $2^k k!$,
which is the centralizer of $t_k$ in $S_{2k}$, where
$t_k \in S_{2k}$ is the product of the transpositions $(1 \ 2), (3 \ 4),\dots,
(2k-1 \ 2k)$.
Let $L(S_{2k},H_k)$ be the subspace of all $H_k$-biinvariant functions in $L(S_{2k})$:
$$
L(S_{2k},H_k)= \{f \in L(S_{2k}) \ | \ f(\zeta \sigma)=f(\sigma \zeta)=f(\sigma)
\quad (\sigma \in S_{2k}, \ \zeta \in H_k)\}.
$$
It is a \emph{commutative} $\mathbb{C}$-algebra under the convolution.

We introduce another product on $L(S_{2k},H_k)$.
For $f_1,f_2 \in L(S_{2k},H_k)$, we define
$$
(f_1 \sharp f_2)(\sigma)= \sum_{\tau \in M_{2k}} f_1(\sigma \tau) f_2(\tau^{-1})
\qquad (\sigma \in S_{2k}).
$$
Note that $f_1 \sharp f_2 = (2^k k!)^{-1} f_1 * f_2$.
In fact, since $M_{2k}$ gives the representative of cosets $\sigma H_{k}$ in $S_{2k}$
and since $f_1,f_2$ are $H_k$-biinvariant, we have
$$
(f_1*f_2)(\sigma)= \sum_{\tau \in M_{2k}} \sum_{\zeta \in H_k}
f_1(\sigma (\tau \zeta)) f_2( (\tau \zeta)^{-1})=
\sum_{\tau \in M_{2k}} \sum_{\zeta \in H_k}
f_1(\sigma \tau) f_2( \tau^{-1})= |H_k| (f_1 \sharp f_2)(\sigma).
$$
The new product $\sharp$ is most of the same as the convolution $*$ on $L(S_{2k})$
up to the normalization factor $2^k k!$,
 but it will be convenient in the present context.
We note that $L(S_{2k},H_k)$ is a commutative algebra under the product $\sharp$ with the
identity element
$$
\mathbf{1}_{H_k}(\sigma)= \begin{cases}1 & \text{if $\sigma \in H_k$} \\
0 & \text{otherwise}.
\end{cases}
$$

Consider the  function $z^{\kappa'(\cdot)}$ with a complex parameter $z$ defined by
$$
S_{2k} \ni \sigma \mapsto z^{\kappa'(\sigma)} \in \mathbb{C},
$$
which belongs to $L(S_{2k},H_k)$.
The \emph{orthogonal Weingarten function}
$\Wg^{\mathrm{O}} (\sigma;z) \ (\sigma \in S_{2k})$
is the unique element in $L(S_{2k},H_k)$ satisfying
$$
z^{\kappa'(\cdot)} \sharp \Wg^{\mathrm{O}}(\cdot;z) \sharp z^{\kappa'(\cdot)} = z^{\kappa'(\cdot)}
\qquad \text{and} \qquad
\Wg^{\mathrm{O}}(\cdot ;z) \sharp z^{\kappa'(\cdot)} \sharp \Wg^{\mathrm{O}}(\cdot ;z)=
\Wg^{\mathrm{O}}(\cdot ;z).
$$

For each partition $\lambda$ of $k$,
the zonal spherical function $\omega^\lambda$ is defined by
$\omega^\lambda = (2^k k!)^{-1} \chi^{2\lambda} * \mathbf{1}_{H_k}$,
where $2\lambda=(2\lambda_1,2\lambda_2,\dots)$,
and the family of $\omega^\lambda$
form a linear basis of $L(S_{2k},H_k)$.
The expansion of $\Wg^{\mathrm{O}}(\cdot;z)$ in terms of $\omega^\lambda$ is given by
$$
\Wg^{\mathrm{O}} (\sigma;z) = \frac{2^k k!}{(2k)!}
\sum_{\begin{subarray}{c} \lambda \vdash k \\ C'_{\lambda}(z) \not= 0 \end{subarray}}
\frac{f^{2\lambda}}{C'_\lambda(z)} \omega^\lambda(\sigma) \qquad (\sigma \in S_{2k}),
$$
summed over all parititons $\lambda$ of $k$ satisfying $C'_\lambda(z) \not =0$,
where
$$
C'_\lambda(z)= \prod_{i=1}^{\ell(\lambda)} \prod_{j=1}^{\lambda_i} (z+2j-i-1).
$$
In particular, if $C'_\lambda(z) \not=0$ for all partitions $\lambda$ of $k$,
functions $z^{\kappa'(\cdot)}$ and $\Wg^{\mathrm{O}}(\cdot ;z)$ are
their inverse of each other and
satisfy $z^{\kappa'(\cdot)} \sharp \Wg^{\mathrm{O}} (\cdot;z) = \mathbf{1}_{H_k}$.

Let $\mathrm{O}(n)$ be the real orthogonal group of degree $n$,
equipped with its Haar probability measure.

\begin{prop}[\cite{CM09}] \label{prop:OrthogonalWeingartenCalculus}
Let $U=(u_{ij})_{1 \le i,j \le n}$ be an $n \times n$ Haar-distributed orthogonal matrix.
For two sequences $\bm{i}=(i_1,\dots,i_{2k})$ and $\bm{j}=(j_1,\dots,j_{2k})$, we have
\begin{equation}
E[u_{i_1 j_1} u_{i_2 j_2} \cdots u_{i_{2k} j_{2k}}] = \sum_{\sigma,\tau \in M_{2k}}
\delta'_\sigma(\bm{i}) \delta'_{\tau}(\bm{j}) \Wg^{\mathrm{O}}(\sigma^{-1} \tau;n).
\end{equation}
Here $\sigma, \tau \in M_{2k}$ are regarded as permutations in $S_{2k}$,
and so is $\sigma^{-1} \tau$.
\end{prop}

We will need the following function later.
Define the function $\Wg^{\mathrm{O}}(\cdot;z,w)$ in $L(S_{2k},H_k)$
with two complex parameters $z,w \in \mathbb{C}$ by
\begin{equation}  \label{eq:double-orthogonalWg}
\Wg^{\mathrm{O}}(\cdot ;z,w) = \Wg^{\mathrm{O}}(\cdot;z) \sharp \Wg^{\mathrm{O}}(\cdot;w).
\end{equation}
More precisely,
$$
\Wg^{\mathrm{O}}(\cdot ;z,w) =
\frac{2^k k!}{(2k)!}
\sum_{\begin{subarray}{c} \lambda \vdash k \\ C'_{\lambda}(z) C'_{\lambda}(w) \not= 0 \end{subarray}}
\frac{f^{2\lambda}}{C'_\lambda(z)C'_\lambda(w)} \omega^\lambda.
$$

\subsection{Wishart matrices and their inverse}

\subsubsection{Complex Wishart matrices}

Let $X$ be an $n \times p$ random matrix
whose columns are i.i.d. complex vectors which follow
$n$-dimensional complex normal distributions
$\mathrm{N}_\mathbb{C}(\bm{0},\Sigma)$, where $\Sigma$ is an $n \times n$ positive definite
Hermitian matrix.
Then we call a random matrix $W= X X^*$ a (centered) \emph{complex Wishart matrix}.

We will need the computation of the local moments for the inverse $W^{-1}$.

\begin{prop}[\cite{GLM}] \label{prop-complex-Wishart-inverse}
Let $W$ be a complex Wishart matrix defined as above.
Put $q=p-n$.
If $\pi \in S_k$ and  $q \ge k$, then
\begin{equation}
E[\Tr_\pi (W^{-1})] = (-1)^k \sum_{\tau \in S_k} \Wg^{\mathrm{U}}(\pi\tau^{-1};-q)
\Tr_\tau(\Sigma^{-1}).
\end{equation}
\end{prop}

\subsubsection{Real Wishart matrices}

Let $X$ be an $n \times p$ random matrix
whose columns are i.i.d. vectors which follow
$n$-dimensional real normal distributions
$\mathrm{N}_\mathbb{R}(\bm{0},\Sigma)$, where $\Sigma$ is an $n \times n$ positive definite
real symmetric matrix.
Then we call a random matrix $W= X \trans{X}$ a (centered) \emph{real Wishart matrix}.

\begin{prop}[\cite{M_Wishart}]  \label{prop-real-Wishart-inverse}
Let $W$ be a real Wishart matrix defined as above.
Put $q=p-n-1$.

If $\pi \in M_{2k}$ and  $q \ge 2k-1$, then
\begin{equation}
E[\Tr_\pi' (W^{-1})] = (-1)^k \sum_{\tau \in M_{2k}} \Wg^{\mathrm{O}}(\pi\tau^{-1};-q)
\Tr_\tau'(\Sigma^{-1}).
\end{equation}
\end{prop}

\section{Invariant random matrices}\label{sec:Invariance}

In this section we consider random matrices with invariance property and establish the link between local and global moments.

\subsection{Conjugacy invariance}

\subsubsection{Unitary case}

\begin{thm}\label{thm:unitary-local-global}
Let $W=(w_{ij})$ be an $n \times n$ complex Hermitian random matrix with the invariant property
such that $UWU^*$ has the same distribution as
$W$ for any unitary matrix $U$.
For two sequences $\bm{i}=(i_1,\dots,i_{k})$
and $\bm{j}=(j_1,\dots,j_k)$, we have
\begin{equation*}
E[w_{i_1 j_1} w_{i_2 j_2} \dots w_{i_{k} j_{k}}]=
\sum_{\sigma,\tau\in S_{k}} \delta_{\sigma}(\bm{i},\bm{j})  \Wg^{\mathrm{U}}(\sigma^{-1} \tau;n) E[\Tr_{\tau}(W)].
\end{equation*}
\end{thm}

Before we prove this theorem we need the following lemma

\begin{lem}\label{lem-unitary-case}
Let $W$ be as in Theorem \ref{thm:unitary-local-global}.
$W$ has the same distribution as $UDU^*$,
 where $U$ is a Haar distributed random unitary matrix,
$D$ is a diagonal matrix whose
eigenvalues have the same distribution as those of $W$, and $D,U$ are independent.
\end{lem}

\begin{proof}
Let $U,D$ be matrices ($U$ unitary, and $D$ diagonal) such that $W=UDU^*$. It is possible to have $U,D$
as measurable functions of $W$
(if the singular values have no multiplicity this follows from the fact that
$U$ can be essentially chosen in a canonical way, and in the general case, it follows by an approximation argument).
So, we may consider that $U,D$ are also random variables and that the sigma-algebra generated by $U,D$ is the same as
the sigma-algebra generated by $W$.

Let $V$ be a deterministic
unitary matrix.
The fact that $VWV^*$ has the same distribution as $W$ and our previous uniqueness
considerations imply that $VU$ has the same distribution as $U$. By uniqueness of the Haar measure, this implies
that $U$ has to be distributed according to the Haar measure.

To conclude the proof, we observe that instead of taking $V$ to be a deterministic unitary matrix, we could
have taken $V$ random, independent from $W$, and distributed according to the Haar measure without changing the
fact that $VWV^*$ has the same distribution as $W$.
This implies that $U$ can be replaced by $VU$, and clearly, $VU$ is Haar distributed, and independent from $D$,
so the proof is complete.
\end{proof}

\begin{proof}[Proof of Theorem \ref{thm:unitary-local-global}]
From Lemma \ref{lem-unitary-case},
each matrix entry $w_{ij}$ has the same distribution as $\sum_{r=1}^n u_{ir} d_{r} \overline{u_{jr}}$,
where $U=(u_{ij})$ and $D=\diag(d_1,\dots,d_n)$ are
unitary and diagonal matrices respectively and
$U,D$ are independent.
It follows that
\begin{align*}
&E[w_{i_1 j_1} w_{i_2 j_2} \cdots w_{i_{k} j_{k}}] \\
=& \sum_{\bm{r}=(r_1, \dots, r_k)} E[d_{r_1} d_{r_2} \cdots d_{r_k}]
\cdot E[u_{i_1 r_1} u_{i_2 r_2} \cdots u_{i_{k} r_k}
\overline{u_{j_1 r_1} u_{j_2 r_2} \cdots u_{j_{k} r_k}}].
\end{align*}
The unitary Weingarten calculus
(Proposition  \ref{prop:UnitaryWeingartenCalculus}) gives
\begin{align*}
=& \sum_{\bm{r}=(r_1, \dots, r_k)} E[d_{r_1} d_{r_2} \cdots d_{r_k}]
\sum_{\sigma, \tau \in S_k}
\delta_{\sigma}(\bm{i},\bm{j})
\delta_{\tau}(\bm{r},\bm{r})
\Wg^{\mathrm{U}}(\sigma^{-1} \tau;n) \\
=& \sum_{\sigma, \tau \in S_k}
\delta_{\sigma}(\bm{i},\bm{j})
\Wg^{\mathrm{U}}(\sigma^{-1} \tau;n)
\sum_{\bm{r}=(r_1, \dots, r_k)}
\delta_{\tau}(\bm{r},\bm{r})
E[d_{r_1} d_{r_2} \cdots d_{r_k}].
\end{align*}
To conclude the proof, we have to show:
For $\tau \in S_k$ and a diagonal matrix $D=\diag(d_1,\dots,d_n)$,
\begin{equation} \label{eq:unitary-combinatorial}
\sum_{\bm{r}=(r_1, \dots, r_k)}
\delta_{\tau}(\bm{r},\bm{r})
d_{r_1} d_{r_2} \cdots d_{r_k} = \Tr_\tau (D).
\end{equation}
We observe that $\delta_{\tau}(\bm{r},\bm{r})$ survives
if and only if all $r_i$ in each cycle of $\tau$ coincide.
Hence,
if $\tau$ has the cycle-type $\mu=(\mu_1,\dots,\mu_l)$, then
$$
\sum_{\bm{r}=(r_1, \dots, r_k)}
\delta_{\tau}(\bm{r},\bm{r})
d_{r_1} d_{r_2} \cdots d_{r_k}=
\sum_{s_1,\dots,s_l} d_{s_1}^{\mu_1} \cdots d_{s_l}^{\mu_l} =
\Tr(D^{\mu_1}) \cdots \Tr(D^{\mu_l}) =
\Tr_\tau(D),
$$
which proves \eqref{eq:unitary-combinatorial}.
\end{proof}

\begin{example}
Let $W$ be as in Theorem \ref{thm:unitary-local-global}.
For each $1 \le i \le n$ and $k \ge 1$,
\begin{equation} \label{eq:ex_unitary_single_entry}
E[w_{ii}^{k}] = \frac{1}{n(n+1) \cdots (n+k-1)}
\sum_{\mu \vdash k} \frac{k!}{z_\mu}
E \left[\prod_{i=1}^{\ell(\mu)} \Tr (W^{\mu_j}) \right]
\end{equation}
summed over all partition $\mu$ of $k$.
Here
$$
z_\mu = \prod_{i \ge 1} i^{m_i(\mu)} \, m_i (\mu)!
$$
with the multiplicities  $m_i(\mu)$ of $i$ in $\mu$.
In fact, Theorem \ref{thm:unitary-local-global} implies
the identity
$E[w_{ii}^{k}] = \sum_{\sigma \in S_k}
\Wg^{\mathrm{U}} (\sigma;n)
\cdot \sum_{\tau \in S_k} E [\Tr_\tau (W)]$,
and the claim therefore is obtained by the following two
known facts:
$$
\sum_{\sigma \in S_k}\Wg^{\mathrm{U}}(\sigma;n)
= \frac{1}{n(n+1) \cdots (n+k-1)};
$$
the number of permutations in $S_k$ of cycle-type $\mu$
is $k!/z_\mu$.
When $k=1$ the equation \eqref{eq:ex_unitary_single_entry} gives a trivial identity
$E[w_{ii}] = \frac{1}{n} E[\Tr (W)]$.
When $k=2,3$, it gives
\begin{align*}
E[w_{ii}^2] =& \frac{1}{n(n+1)} (E[\Tr(W^2)] +E[\Tr (W)^2]); \\
E[w_{ii}^3] =& \frac{1}{n(n+1)(n+2)}
(2E[\Tr(W^3)] +3 E[\Tr(W^2) \Tr(W)]
+E[\Tr (W)^3]).
\end{align*}
\end{example}

\subsubsection{Orthogonal case}

\begin{thm}\label{thm:orthogonal-local-global}
Let $W=(w_{ij})$ be an $n \times n$ real symmetric
random matrix with the invariant property
such that $UW\trans{U}$ has the same distribution as
$W$ for any orthogonal matrix $U$.
For any sequence $\bm{i}=(i_1,\dots,i_{2k})$, we have
\begin{equation*}
E[w_{i_1 i_2} w_{i_3 i_4} \dots w_{i_{2k-1} i_{2k}}]=
\sum_{\sigma,\tau\in M_{2k}} \delta'_{\sigma}(\bm{i})  \Wg^{\mathrm{O}}(\sigma^{-1} \tau;n) E[\Tr_{\tau}'(W)].
\end{equation*}
\end{thm}

\begin{proof}
As in Lemma \ref{lem-unitary-case},
$W$ has the same distribution $UD \trans{U}$,
where $U=(u_{ij})$ and $D=\diag(d_1,\dots,d_n)$ are
orthogonal and diagonal matrices respectively and
$U,D$ are independent.
We have
\begin{align*}
&E[w_{i_1 i_2} w_{i_3 i_4} \dots w_{i_{2k-1} i_{2k}}] \\
=&
\sum_{\bm{r}=(r_1, \dots, r_k)} E[d_{r_1} d_{r_2} \cdots d_{r_k}]
\cdot E[u_{i_1 r_1} u_{i_2 r_1} u_{i_3 r_2} u_{i_4 r_2} \cdots u_{i_{2k-1} r_k} u_{i_{2k}r_k}],
\end{align*}
and the orthogonal Weingarten calculus (Proposition \ref{prop:OrthogonalWeingartenCalculus}) gives
\begin{align*}
=& \sum_{\bm{r}=(r_1, \dots, r_k)} E[d_{r_1} d_{r_2} \cdots d_{r_k}]
\sum_{\sigma, \tau \in M_{2k}}
\delta_{\sigma}'(\bm{i}) \delta_{\tau}'(\tilde{\bm{r}})
\Wg^{\mathrm{O}}(\sigma^{-1} \tau;n) \\
=& \sum_{\sigma, \tau \in M_{2k}}
\delta_{\sigma}'(\bm{i})
\Wg^{\mathrm{O}}(\sigma^{-1} \tau;n)
\sum_{\bm{r}=(r_1, \dots, r_k)}
\delta_{\tau}'(\tilde{\bm{r}})
E[d_{r_1} d_{r_2} \cdots d_{r_k}],
\end{align*}
where $\tilde{\bm{r}}=(r_1,r_1,r_2,r_2,\dots,r_k,r_k)$ for
each $\bm{r}=(r_1,r_2,\dots,r_k)$.

Recall notation defined in section \ref{sec:NotationR}.
To conclude the proof, we have to show:
For $\tau \in S_{2k}$ and a diagonal matrix $D=\diag(d_1,\dots,d_n)$,
\begin{equation} \label{eq:orthogonal-combinatorial}
\sum_{\bm{r}=(r_1, \dots, r_k)}
\delta_{\tau}'(\tilde{\bm{r}})
d_{r_1} d_{r_2} \cdots d_{r_k} = \Tr_\tau' (D).
\end{equation}
This equation follows from the following fact:

 $\delta_{\tau}'(\tilde{\bm{r}})$ survives
if and only if all $r_i$ in each component of the graph
$\Gamma(\tau)$ coincide.
\end{proof}

\begin{example}
Let $W$ be as in Theorem \ref{thm:orthogonal-local-global}.
For each $1 \le i \le n$ and $k \ge 1$,
\begin{equation} \label{eq:ex_orthogonal_single_entry}
E[w_{ii}^{k}] = \frac{1}{n(n+2) \cdots (n+2k-2)}
\sum_{\mu \vdash k} \frac{2^k k!}{2^{\ell(\mu)} z_\mu}
E \left[\prod_{i=1}^{\ell(\mu)} \Tr (W^{\mu_j}) \right].
\end{equation}
In fact, Theorem \ref{thm:orthogonal-local-global} with the following two
facts gives the claim:
$$
\sum_{\sigma \in M_{2k} }\Wg^{\mathrm{O}}(\sigma;n)
= \frac{1}{n(n+2) \cdots (n+2k-2)};
$$
the number of pair partitions in $M_{2k}$ of coset-type $\mu$
is $2^k k!/(2^{\ell(\mu)}z_\mu)$.
When $k=2,3$, \eqref{eq:ex_orthogonal_single_entry} gives
\begin{align*}
E[w_{ii}^2] =& \frac{1}{n(n+2)} (2E[\Tr(W^2)] +E[\Tr (W)^2]); \\
E[w_{ii}^3] =& \frac{1}{n(n+2)(n+4)}
(8E[\Tr(W^3)] +6 E[\Tr(W^2) \Tr(W)]
+E[\Tr (W)^3]).
\end{align*}
\end{example}

\subsection{Left-right invariance}

\subsubsection{Unitary case}

\begin{thm}  \label{thm:unitary-LRinvariance}
Let $X$ be a complex $n \times p$ random matrix
which has the same distribution as $UXV$ for
any unitary matrices $U,V$.
For four sequences
$\bm{i}=(i_1,\dots,i_{k})$, $\bm{j}=(j_1,\dots,j_{k})$,
$\bm{i}'=(i_1',\dots,i_k')$, $\bm{j}'=(j_1',\dots,j_k')$,
$$
E[x_{i_1 j_1} \cdots x_{i_k j_k}
\overline{x_{i_1'j_1'} \cdots x_{i_k' j_k'}}] =
\sum_{\sigma_1,\sigma_2,\pi \in S_k}
\delta_{\sigma_1} (\bm{i},\bm{i}')
\delta_{\sigma_2} (\bm{j},\bm{j}') \Wg^{\mathrm{U}}(\pi \sigma_1^{-1} \sigma_2;n,p)
E[\Tr_{\pi} (X X^*)],
$$
where $\Wg^{\mathrm{U}}(\cdot;n,p)$ is defined in \eqref{eq:double-unitaryWg}.
\end{thm}

\begin{proof}
As in Lemma \ref{lem-unitary-case},
we can see that
$X$ has the same distribution $UD V^*$,
where $U$ and $V$ are Haar distributed $n \times n$
and $p \times p$ random unitary matrices, respectively,
and $D$ is an $n \times p$ diagonal matrix whose singular values have the same distribution as those of $X$.
Moreover, $D, U, V$ are independent.

Since each entry $x_{ij}$ has the same distribution as
$\sum_{r=1}^{\min (n,p)} u_{ir} d_r \overline{v_{jr}}$,
it follows from the independence of $U$, $D$, and $V$ that
\begin{align*}
&E[x_{i_1 j_1} \cdots x_{i_k j_k}
\overline{x_{i_1'j_1'} \cdots x_{i_k' j_k'}}
] \\
=& \sum_{\bm{r}=(r_1,\dots,r_k)}
\sum_{\bm{r'}=(r_1',\dots,r_k')}
E[d_{r_1} \cdots d_{r_k} \overline{d_{r_1'} \cdots d_{r_k'}}] \\
& \qquad
\times E[u_{i_1 r_1} \cdots u_{i_k r_k} \overline{u_{i_1'r_1'}
\cdots u_{i_k' r_k'}}] \times
E[\overline{v_{j_1 r_1} \cdots v_{j_k r_k}} v_{j_1'r_1'}
\cdots v_{j_k' r_k'}].
\end{align*}
Here $r_s,r_s'$ run over $1,2,\dots, \min (p,n)$.
From
the unitary Weingarten calculus
(Proposition  \ref{prop:UnitaryWeingartenCalculus}),
we have
\begin{align}
=& \sum_{\sigma_1,\tau_1,\sigma_2,\tau_2 \in S_k}
\delta_{\sigma_1}(\bm{i},\bm{i}') \delta_{\sigma_2}(\bm{j},\bm{j}')  \Wg^{\mathrm{U}}(\sigma_1^{-1}\tau_1;n)
\Wg^{\mathrm{U}}(\sigma_2^{-1}\tau_2;p)  \notag \\
& \times
\sum_{\bm{r}=(r_1,\dots,r_k)}
\sum_{\bm{r'}=(r_1',\dots,r_k')} \delta_{\tau_1} (\bm{r},
\bm{r}')\delta_{\tau_2} (\bm{r}, \bm{r}')
E[d_{r_1} \cdots d_{r_k} \overline{d_{r_1'} \cdots d_{r_k'}}].
\label{eq:unitary-LR1}
\end{align}
Since $\delta_{\tau_1} (\bm{r},\bm{r}') \delta_{\tau_2} (\bm{r},\bm{r}') =1$
if and only if $r_s'=r_{\tau_2(s)} \ (1 \le s \le k)$ and
$\delta_{\tau_1^{-1}\tau_2} (\bm{r},\bm{r})=1$,
we have
\begin{align*}
&\sum_{\bm{r}=(r_1,\dots,r_k)}
\sum_{\bm{r'}=(r_1',\dots,r_k')} \delta_{\tau_1} (\bm{r},
\bm{r}')\delta_{\tau_2} (\bm{r}, \bm{r}')
d_{r_1} \cdots d_{r_k} \overline{d_{r_1'} \cdots d_{r_k'}} \\
=& \sum_{\bm{r}=(r_1,\dots,r_k)}
\delta_{\tau_1^{-1}\tau_2} (\bm{r},\bm{r})
d_{r_1} \cdots d_{r_k} \overline{d_{r_1} \cdots d_{r_k}},
\end{align*}
which equals $\Tr_{\tau_1^{-1} \tau_2} (D D^*)$ by
\eqref{eq:unitary-combinatorial}.
Substituting this fact into \eqref{eq:unitary-LR1}, we have
\begin{align*}
E[x_{i_1 j_1} \cdots x_{i_k j_k}
\overline{x_{i_1'j_1'} \cdots x_{i_k' j_k'}}
]
=&
\sum_{\sigma_1,\sigma_2\in S_k}
\delta_{\sigma_1}(\bm{i},\bm{i}') \delta_{\sigma_2}(\bm{j},\bm{j}') \\
&\times \sum_{\tau_1,\tau_2 \in S_k}
\Wg^{\mathrm{U}}(\sigma_1^{-1}\tau_1;n)
\Wg^{\mathrm{U}}(\sigma_2^{-1}\tau_2;p)E[\Tr_{\tau_1^{-1} \tau_2} (X X^*)].
\end{align*}
The proof of the theorem is follows from the following observation.
\begin{align*}
&\sum_{\tau_1,\tau_2 \in S_k}
\Wg^{\mathrm{U}}(\sigma_1^{-1}\tau_1;n) \Wg^{\mathrm{U}}(\sigma_2^{-1}\tau_2;p)
\Tr_{\tau_1^{-1} \tau_2} (X X^*) \\
=& \sum_{\tau_2, \pi \in S_k}
\Wg^{\mathrm{U}}(\sigma_1^{-1} \tau_2 \pi;n)
\Wg^{\mathrm{U}}(\sigma_2^{-1} \tau_2;p) \Tr_{\pi^{-1}} (X X^*)
\qquad (\because \tau_1=\tau_2 \pi) \\
=&
\sum_{\tau_2, \pi \in S_k}
\Wg^{\mathrm{U}}(\pi \sigma_1^{-1} \tau_2;n)
\Wg^{\mathrm{U}}(\tau_2^{-1}\sigma_2;p) \Tr_{\pi^{-1}} (X X^*)
\qquad (\because \Wg^{\mathrm{U}}(\sigma;z)= \Wg^{\mathrm{U}}(\sigma^{-1};z))
\\
=&
\sum_{\pi \in S_k} \Wg^{\mathrm{U}}(\pi \sigma_1^{-1} \sigma_2;n,p)
E[\Tr_{\pi} (X X^*)].
\end{align*}
At the last equality we have used the definition of
$\Wg^{\mathrm{U}}(\cdot;n,p)$.
\end{proof}

\begin{example}
If $X$ satisfies the condition of
Theorem \ref{thm:unitary-LRinvariance}, we have
$$
E[x_{i j} \overline{x_{i' j'}}] = \delta_{i i'} \delta_{j j'}
\frac{1}{n p} E[\Tr (X X^*)].
$$
\end{example}

\subsubsection{Orthogonal case}

\begin{thm} \label{thm:orthogonal-LRinvariance}
Let $X$ be a real $n \times p$ random matrix
which has the same distribution as $UXV$ for
any orthogonal matrices $U,V$.
For two sequences
$\bm{i}=(i_1,\dots,i_{2k})$ and $\bm{j}=(j_1,\dots,j_{2k})$,
$$
E[x_{i_1 j_1} \cdots x_{i_{2k} j_{2k}}]
= \sum_{\sigma_1,\sigma_2,\pi \in M_{2k}}
\delta_{\sigma_1}' (\bm{i})
\delta_{\sigma_2}' (\bm{j}) \Wg^{\mathrm{O}}(\pi \sigma_1^{-1} \sigma_2;n,p)
E[\Tr_{\pi} (X \trans{X})],
$$
where $\Wg^{\mathrm{O}}(\cdot;n,p)$ is defined in \eqref{eq:double-orthogonalWg}.
\end{thm}
\begin{proof}
In a similar way to the proof of Theorem \ref{thm:unitary-LRinvariance},
we have
\begin{align*}
&E[x_{i_1 j_1} \cdots x_{i_{2k} j_{2k}}] \\
=& \sum_{\sigma_1,\sigma_2,\tau_1,\tau_2 \in M_{2k}}
\delta_{\sigma_1}' (\bm{i})
\delta_{\sigma_2}' (\bm{j})
\Wg^{\mathrm{O}}(\sigma_1^{-1} \tau_1;n)
\Wg^{\mathrm{O}}(\sigma_2^{-1} \tau_2;p) \\
& \times \sum_{\bm{r}=(r_1,\dots,r_{2k})}
\delta_{\tau_1}' (\bm{r}) \delta_{\tau_2}' (\bm{r})
E[d_{r_1} \cdots d_{r_{2k}}].
\end{align*}
We observe that
$\delta_{\tau_1}' (\bm{r}) \delta_{\tau_2}' (\bm{r})=1$
if and only if
all $r_i$ in each component of $\Gamma(\tau^{-1}_1\tau_2)$ coincide.
Letting $(\mu_1,\dots,\mu_l)$ to be a coset-type of
$\tau_1^{-1}\tau_2$ we have
$$
\sum_{\bm{r}=(r_1,\dots,r_{2k})}
\delta_{\tau_1}' (\bm{r}) \delta_{\tau_2}' (\bm{r})
d_{r_1} \cdots d_{r_{2k}}
= \sum_{s_1,\dots,s_l} d_{s_1}^{2\mu_1} \cdots d_{s_l}^{2\mu_l}
= \Tr_{\tau_1^{-1} \tau_2}'(D \trans{D})
= \Tr_{\tau_1^{-1} \tau_2}'(X \trans{X}).
$$
We thus have proved
$$
E[x_{i_1 j_1} \cdots x_{i_{2k} j_{2k}}]
= \sum_{\sigma_1,\sigma_2,\tau_1,\tau_2 \in M_{2k}}
\delta_{\sigma_1}' (\bm{i})
\delta_{\sigma_2}' (\bm{j})
\Wg^{\mathrm{O}}(\sigma_1^{-1} \tau_1;n)
\Wg^{\mathrm{O}}(\sigma_2^{-1} \tau_2;p) \Tr'_{\tau_1^{-1}\tau_2} (X\trans{X}).
$$
The remaining step is shown in a similar way to
the proof of Theorem \ref{thm:unitary-LRinvariance}.
(Replace a sum $\sum_{\sigma \in M_{2k}}$ by $(2^k k!)^{-1}\sum_{\sigma \in S_{2k}}$.)
\end{proof}

\begin{example}
If $X$ satisfies the condition of
Theorem \ref{thm:orthogonal-LRinvariance}, we have
$$
E[x_{i_1 j_1} x_{i_2 j_2}] = \delta_{i_1 i_2} \delta_{j_1 j_2}
\frac{1}{n p} E[\Tr (X \trans{X})].
$$
\end{example}

\section{Application to statistics} \label{sec:application}

\subsection{Pseudo-inverse of a Ginibre matrix}

\subsubsection{Complex case}

An $n \times p$ complex
Ginibre matrix $G$ is a random matrix
whose columns are i.i.d. and distributed as
$n$-dimensional normal distribution
$\mathrm{N}_\mathbb{C}(0,\Sigma)$,
where $\Sigma$ is an $n \times n$ positive definite Hermitian
matrix.
If $G=UDV^*$ is a singular value decomposition of $G$,
the matrix $G^-=VD^{-}U^*$ is the pseudo-inverse of $G$, where $D^{-}$ is the $p\times n$
diagonal obtained by inverting point wise the entries of $D$ along the diagonal (and zero if the diagonal entry is zero).

Note that it is easy to check that the pseudo-inverse is well-defined in the sense that it does not depend on the
decomposition $G=UDV^*$. Actually, in the same vein as in section
\ref{sec:Wg} where the pseudo-inverse is introduced in the context of Weingarten functions,
the properties $GG^-G=G$, $G^-GG^-=G^-$ together with the fact that $GG^-,G^-G$ are selfadjoint,
suffice to define uniquely the inverse.
If the matrix is invertible, the pseudo-inverse is the inverse (this notion of pseudo-inverse is sometimes known
as the Moore-Penrose pseudo inverse).


\begin{thm} \label{thm:complex-Ginibre-inverse}
Let $G^{-}=(g^{ij})$ be the pseudo-inverse matrix
of an $n \times p$ complex Ginibre matrix
associated with an $n \times n$ positive definite Hermitian matrix $\Sigma$.
Put $q=p-n$ and suppose $n, q \ge k$.
For four sequences $\bm{i}=(i_1,\dots,i_k)$,
$\bm{j}=(j_1,\dots,j_k)$, $\bm{i}'=(i_1',\dots,i_k')$, and
$\bm{j}'=(j_1',\dots,j_k')$, we have
\begin{align*}
& E[ g^{i_1 j_1} \cdots g^{i_k j_k} \overline{g^{i_1' j_1'} \cdots g^{i_k' j_k'}}] \\
=& (-1)^k \sum_{\sigma,\rho \in S_k}
 \delta_{\sigma}(\bm{i},\bm{i}')  \Wg^{\mathrm{U}}(\sigma^{-1}\rho ;p,-q)
 \overline{
(\Sigma^{-1})_{j_{\rho(1)} j_1'} \cdots  (\Sigma^{-1})_{j_{\rho(k)} j_k'} },
\end{align*}
where $\Wg^{\mathrm{U}}(\cdot ;p,-q)$ is defined in \eqref{eq:double-unitaryWg}.
\end{thm}

\begin{proof}
Let $Z$ be an $n \times p$  matrix of i.i.d.
$\mathrm{N}_\mathbb{C}(0,1)$ random variables.
Then it is immediate to see that
$\Sigma^{1/2} Z$ has the same distribution as $G$.
Therefore each $g^{ij}$ has the same distribution as $\sum_{r=1}^n z^{ir} (\Sigma^{-1/2})_{rj}$,
where $Z^{-}=(z^{ij})_{1 \le i \le p, 1 \le j \le n}$ is the pseudo-inverse matrix of $Z$, and hence
\begin{align*}
 E[ g^{i_1 j_1} \cdots g^{i_k j_k} \overline{g^{i_1' j_1'} \cdots g^{i_k' j_k'}}]
 =&\sum_{\bm{r}=(r_1,\dots,r_k)} \sum_{\bm{r}'=(r_1',\dots,r_k')}
 \prod_{s=1}^k (\Sigma^{-1/2})_{r_s j_s} \overline{(\Sigma^{-1/2})_{r_s' j_s'}} \\
  & \qquad \times E[z^{i_1 r_1} \cdots z^{i_k r_k}
  \overline{z^{i_1' r_1'} \cdots z^{i_k' r_k'}}].
\end{align*}
Since $Z^{-}$ is a $p \times n$ matrix satisfying the condition on Theorem \ref{thm:unitary-LRinvariance}, we have
\begin{align*}
& E[z^{i_1 r_1} \cdots z^{i_k r_k}
  \overline{z^{i_1' r_1'} \cdots z^{i_k' r_k'}}] \\
  =& \sum_{\sigma,\rho,\pi \in S_k}
\delta_{\sigma} (\bm{i},\bm{i}')
\delta_{\rho} (\bm{r},\bm{r}')
\Wg^{\mathrm{U}}(\sigma^{-1} \pi \rho;p,n)
E[\Tr_{\pi} (Z^{-} (Z^{-})^*)].
\end{align*}
Moreover, from the condition of $q=p-n \geq k$,
we can apply Proposition \ref{prop-complex-Wishart-inverse}
with
$W=Z Z^*$,  and
$$
E[\Tr_{\pi} (Z^{-} (Z^{-})^*)] =
E[\Tr_{\pi} (W^{-1})] =
(-1)^k \sum_{\tau \in S_k} \Wg^{\mathrm{U}}(\pi \tau^{-1};-q) \Tr_{\tau}(
I_n),
$$
where $I_n$ is the $n \times n$ identity matrix.
Note that $\Tr_\tau(I_n)= n^{\kappa(\tau)}$.
Hence we have obtained
\begin{align*}
&E[ g^{i_1 j_1} \cdots g^{i_k j_k} \overline{g^{i_1' j_1'} \cdots g^{i_k' j_k'}}] \\
=& (-1)^k \sum_{\sigma,\rho, \pi, \tau \in S_k}
\delta_{\sigma} (\bm{i},\bm{i}')
 n^{\kappa(\tau)}
\Wg^{\mathrm{U}}(\sigma^{-1} \pi \rho ;p,n) \Wg^{\mathrm{U}}(\pi^{-1} \tau;-q) \\
& \qquad \times
 \sum_{\bm{r}=(r_1,\dots,r_k)} \sum_{\bm{r}'=(r_1',\dots,r_k')}\delta_{\rho} (\bm{r},\bm{r}')
 \prod_{s=1}^k (\Sigma^{-1/2})_{r_s j_s} \overline{(\Sigma^{-1/2})_{r_s' j_s'}}.
\end{align*}

A direct calculation gives
\begin{align*}
&\sum_{\pi, \tau \in S_k}
 n^{\kappa(\tau)}
\Wg^{\mathrm{U}}(\sigma^{-1} \pi \rho ;p,n) \Wg^{\mathrm{U}}(\pi^{-1} \tau;-q) \\
=& \sum_{\pi, \tau \in S_k} \Wg^{\mathrm{U}}(\rho\sigma^{-1} \pi  ;p,n) \Wg^{\mathrm{U}}(\pi^{-1} \tau;-q)  n^{\kappa(\tau^{-1})} \\
=& [\Wg^{\mathrm{U}}(\cdot  ;p) * \Wg^{\mathrm{U}}(\cdot  ;n) * \Wg^{\mathrm{U}}(\cdot  ;-q)
* n^{\kappa(\cdot)}] (\rho \sigma^{-1}).
\end{align*}
Since $n^{\kappa(\cdot)} * \Wg^{\mathrm{U}}(\cdot  ;n)=\delta_e$ when $n \ge k$,
we have
$$
\sum_{\pi, \tau \in S_k}
 n^{\kappa(\tau)}
\Wg^{\mathrm{U}}(\sigma^{-1} \pi \rho ;p,n) \Wg^{\mathrm{U}}(\pi^{-1} \tau;-q)
=  \Wg^{\mathrm{U}}(\sigma^{-1} \rho;p,-q).
$$

On the other hand, it is easy to see that
\begin{align*}
&\sum_{\bm{r}=(r_1,\dots,r_k)} \sum_{\bm{r}'=(r_1',\dots,r_k')}\delta_{\rho} (\bm{r},\bm{r}')
 \prod_{s=1}^k (\Sigma^{-1/2})_{r_s j_s} \overline{(\Sigma^{-1/2})_{r_s' j_s'}} \\
 =&
 \sum_{r_1,\dots, r_k} \prod_{s=1}^k \overline{(\Sigma^{-1/2})_{j_s r_s}
 (\Sigma^{-1/2})_{r_{\rho(s)}j_s'}}
 =  \sum_{r_1,\dots, r_k} \prod_{s=1}^k
 \overline{(\Sigma^{-1/2})_{j_{\rho(s)} r_{\rho(s)}}
 (\Sigma^{-1/2})_{r_{\rho(s)} j_s'}} \\
 =& \prod_{s=1}^k \( \sum_{r}\overline{
 (\Sigma^{-1/2})_{j_{\rho(s)} r} (\Sigma^{-1/2})_{r j_s'}} \)
 = \prod_{s=1}^k \overline{(\Sigma^{-1})_{j_{\rho(s)} j_s'}}.
\end{align*}

We thus have completed the proof of the theorem.
\end{proof}

\begin{example}
For $G$ given as in Theorem \ref{thm:complex-Ginibre-inverse},
$$
E[g^{ij} \overline{g^{i' j'}}]= \delta_{i,i'}\frac{1}{p(p-n)} \overline{(\Sigma^{-1})_{jj'}}.
$$
\end{example}

\subsubsection{Real case}

An $n \times p$ real
Ginibre matrix $G$ is a random matrix
whose columns are i.i.d. and distributed as
$n$-dimensional normal distribution
$\mathrm{N}_\mathbb{R}(0,\Sigma)$,
where $\Sigma$ is an $n \times n$ positive definite real symmetric
matrix.

\begin{thm}  \label{thm:real-Ginibre-inverse}
Let $G^{-}=(g^{ij})$ be the pseudo-inverse matrix
of an $n \times p$ real Ginibre matrix
associated with an $n \times n$ positive definite real symmetric matrix $\Sigma$.
Put $q=p-n-1$ and suppose $n \ge k$ and $q \ge 2k-1$.
For two sequences $\bm{i}=(i_1,\dots,i_{2k})$ and
$\bm{j}=(j_1,\dots,j_{2k})$, we have
$$
 E[ g^{i_1 j_1} g^{i_2 j_2} \cdots g^{i_{2k} j_{2k}}]
= (-1)^k \sum_{\sigma,\rho \in M_{2k}}
\delta_{\sigma}' (\bm{i})
\Wg^{\mathrm{O}}(\sigma^{-1} \rho;p,-q)
\prod_{\{a,b\} \in \rho} (\Sigma^{-1})_{j_a j_b},
$$
where $\Wg^{\mathrm{O}}(\cdot ;p,-q)$ is defined in \eqref{eq:double-orthogonalWg}.
\end{thm}

\begin{proof}
The proof is similar to that of the complex case if we use
Theorem \ref{thm:orthogonal-LRinvariance}, Proposition \ref{prop-real-Wishart-inverse},
and the following identity: for each $\sigma \in M_{2k}$,
\begin{equation}
\sum_{\bm{r}=(r_1,\dots,r_{2k})} \delta_{\sigma}' (\bm{r}) \prod_{s=1}^{2k} (\Sigma^{-1/2})_{r_sj_s} = \prod_{\{a,b\} \in \sigma} (\Sigma^{-1})_{j_aj_b},
\end{equation}
which is verified easily.
\end{proof}

\begin{example}
For $G$ given as in Theorem \ref{thm:real-Ginibre-inverse},
$$
E[g^{i_1 j_1} g^{i_2 j_2}]=
\delta_{i_1,i_2} \frac{1}{p(p-n-1)} (\Sigma^{-1})_{j_1j_2}.
$$
\end{example}

\subsection{Inverse of compound Wishart matrix}

\subsubsection{Complex case}

Let $\Sigma$ be an $n \times n$ positive definite Hermitian matrix and
let $B$ be a $p \times p$ complex matrix.
Let $Z$ be an $n \times p$ matrix of i.i.d. $\mathrm{N}_\mathbb{C}(0,1)$ random variables.
Then
we call a matrix
$$
W= \Sigma^{1/2} Z B Z^* \Sigma^{1/2}
$$
a complex compound Wishart matrix with shape parameter $B$ and scale parameter $\Sigma$, where $\Sigma^{1/2}$ is the hermitian root of $\Sigma$.

\begin{remark}
If $\Sigma=I_n$, then the corresponding compound Wishart matrix is called white (or standard) compound Wishart.
If $B$ is a positive-definite matrix, then the corresponding compound Wishart matrix can be considered as a
sample covariance matrix under correlated sampling as explained in \cite{bjjnpz}.
\end{remark}

\begin{thm} \label{thm:complex-compound-inverse}
Let $\Sigma$ be an $n \times n$ positive definite Hermitian matrix and
$B$ be a $p \times p$ complex matrix.
 
Let $W^{-1}=(w^{ij})$ be the inverse matrix
of an $n \times n$ complex compound Wishart matrix
with shape parameter $B$ and scale parameter $\Sigma$.
Put $q=p-n$ and suppose $n, q \ge k$.
For two sequences $\bm{i}=(i_1,\dots,i_{k})$ and
$\bm{j}=(j_1,\dots,j_k)$, we have
$$
 E[ w^{i_1 j_1} \cdots w^{i_k j_k} ]
= (-1)^k \sum_{\sigma, \rho \in S_k}
 \Tr_{\sigma} (B^{-}) \Wg^{\mathrm{U}}(\sigma^{-1} \rho;p,-q)
 \overline{
(\Sigma^{-1})_{j_{\rho(1)} j_1'} \cdots  (\Sigma^{-1})_{j_{\rho(k)} j_k'} }.
$$
\end{thm}

\begin{proof}
The matrix $W$ has the same distribution as $GBG^*$, where
$G$ is an $n \times p$ Ginibre matrix associated with $\Sigma$.
If we write $B^{-}=(b^{ij})$ and $G^-=(g^{ij})$,
then
\begin{align*}
E[ w^{i_1 j_1} \cdots w^{i_k j_k} ] =&
\sum_{\bm{r}=(r_1,\dots,r_k)} \sum_{\bm{r}=(r_1',\dots,r_k')}
b^{r_1 r_1'} \cdots b^{r_k r_k'}
E[\overline{g^{r_1 i_1} \cdots g^{r_k i_k}} g^{r_1' j_1} \cdots g^{r_k' j_k}].
\end{align*}

Moreover,
it follows from Theorem \ref{thm:complex-Ginibre-inverse} that
\begin{align*}
E[ w^{i_1 j_1} \cdots w^{i_k j_k} ]=& (-1)^k \sum_{\sigma,\rho \in S_{k}} \Wg^{\mathrm{U}}(\sigma^{-1}\rho;p,-q)
\overline{
(\Sigma^{-1})_{j_{\rho(1)} j_1'} \cdots  (\Sigma^{-1})_{j_{\rho(k)} j_k'} }\\
& \times
\sum_{\bm{r}=(r_1,\dots,r_k)} \sum_{\bm{r}=(r_1',\dots,r_k')}
\delta_{\sigma}(\bm{r},\bm{r}')b^{r_1 r_1'} \cdots b^{r_k r_k'}.
\end{align*}
We finally observe that
$$
\sum_{\bm{r}=(r_1,\dots,r_k)} \sum_{\bm{r}=(r_1',\dots,r_k')}
\delta_{\sigma}(\bm{r},\bm{r}')b^{r_1 r_1'} \cdots b^{r_k r_k'}
= \sum_{\bm{r}=(r_1,\dots,r_k)}
b^{r_1 r_{\sigma(1)}} \cdots b^{r_k  r_{\sigma(k)}}
=\Tr_\sigma(B^-).
$$
\end{proof}

\begin{remark}
If $\Sigma=I_n$ (in the white compound Wishart case), one can observe that a simplification occurs
in the above formula.

In turn, this simplification has the following probabilistic explanation:
the joint distribution of the traces of $W,W^2,\ldots $ has the same law
 as the joint distribution of the traces of $\tilde W, \tilde W^2,\ldots$, where $\tilde W$ is a (non-compound)
 Wishart distribution of scale parameter $B^{1/2}$.
  
 Therefore we can use existing results for the inverse of non-compound Wishart matrices in order to work out this
 case.
\end{remark}

\subsubsection{Real case}

\begin{thm} \label{thm:real-compound-inverse}
Let $\Sigma$ be an $n \times n$ positive definite real symmetric matrix and
$B$ a $p \times p$ real matrix.
Let $W^{-1}=(w^{ij})$ be the inverse matrix
of an $n \times n$ real compound Wishart matrix
with shape parameter $B$ and scale parameter $\Sigma$.
Put $q=p-n-1$ and suppose $n \ge k$ and $q \ge 2k-1$.
For any sequence $\bm{i}=(i_1,\dots,i_{2k})$, we have
$$
 E[ w^{i_1 i_2} \cdots w^{i_{2k-1} i_{2k}} ]
= (-1)^k \sum_{\sigma, \rho \in M_{2k}}
 \Tr'_{\sigma} (B^{-}) \Wg^{\mathrm{O}}(\sigma^{-1} \rho;p,-q)
\prod_{\{u,v \} \in \rho} (\Sigma^{-1})_{i_u,i_v}.
$$
\end{thm}

\begin{proof}
The proof is parallel to the complex case.
We shall leave it to the readers.
\end{proof}

\section*{Acknowledgments}

B.C. was supported by the ANR Granma grant.
B.C. and N.S. were supported by an NSERC discovery grant and an Ontario's ERA grant.
S.M. was supported by a Grant-in-Aid for Young Scientists (B) No.
22740060.

B.C. would like to acknowledge the hospitality of RIMS where he had a visiting position during the fall 2011,
and where this work was initiated.


\bigskip

\noindent
\textsc{Beno\^\i t Collins} \\
D\'epartement de Math\'ematique et Statistique, Universit\'e d'Ottawa,
585 King Edward, Ottawa, ON, K1N6N5 Canada \\
and \\
CNRS, Institut Camille Jordan Universit\'e  Lyon 1, 43 Bd du 11 Novembre 1918, 69622 Villeurbanne,
France \\
\verb|bcollins@uottawa.ca|

\bigskip

\noindent
\textsc{Sho Matsumoto} \\
Graduate School of Mathematics, Nagoya University, Nagoya, 464-8602, Japan. \\
\verb|sho-matsumoto@math.nagoya-u.ac.jp|

\bigskip

\noindent
\textsc{Nadia Saad} \\
D\'epartement de Math\'ematique et Statistique, Universit\'e d'Ottawa,
585 King Edward, Ottawa, ON, K1N6N5 Canada \\
\verb|nkotb087@uottawa.ca|


\begin{thebibliography}{AA11}
%
%
\bibitem[BJJNPZ]{bjjnpz}
Z. Burda, A. Jarosz, J. Jurkieewicz, M. A. Nowak, G. Papp, and I. Zahed,
Applying free random variables to random matrix analysis of financial data.
arXiv.org:physics/0603024, 2006.

\bibitem[CW]{CarvWest}
C. M. Carvalho, M. West,
Dynamic matrix-variate graphical models.
Bayesian Analysis, 2, 2007.

\bibitem[CMW]{CarvMasWest}
C. M. Carvalho, H., Massam, M. West,
Simulation of hyper-inverse Wishart distributions in graphical models.
Biometrika, 94, 2007.

\bibitem[C]{Co03}
B. Collins,
Moments and cumulants of polynomial random variables on unitary groups, the Itzykson-Zuber integral, and free probability.  Int. Math. Res. Not. (2003), no. 17, 953--982.

\bibitem[CM]{CM09}
B. Collins  and S. Matsumoto,
On some properties of orthogonal Weingarten functions.
J. Math. Phys. {\bf 50} (2009), no. 11, 113516, 14 pp.

\bibitem[CS]{CS06}
B. Collins and P. \'{S}niady,
Integration with respect to the Haar measure on unitary, orthogonal and symplectic group.
Comm. Math. Phys. {\bf 264} (2006), no. 3, 773--795.

\bibitem[GLM]{GLM}
P. Graczyk, P. Letac, and H. Massam,
The complex Wishart distribution and the symmetric group.
Ann. Statist.  {\bf 31}  (2003),  no. 1, 287--309.

\bibitem[HP]{hiai-petz}
F. Hiai, and D. Petz,
The semicircle law, free random variables and entropy.
American Mathematical Society, Providence, RI, 2000, vol. 77 of Mathematical Surveys and Monographs.

\bibitem[Mac]{Macdonald}
 I. G. Macdonald, Symmetric Functions nad Hall Polynomials, 2nd ed.,
 Oxford University Press, Oxford (1995).

\bibitem[M1]{M_ortho}
S. Matsumoto,
Jucys-Murphy elements, orthogonal matrix integrals, and Jack measures,
The Ramanujan J. {\bf 26} (2011), 69--107.

\bibitem[M2]{M_Wishart}
S. Matsumoto,
General moments of the inverse real Wishart distribution and orthogonal Weingarten functions.
J. Theoret. Probab., Online First,
DOI 10.1007/s10959-011-0340-0, 25 pp.
arXiv1004.4717v3.

\bibitem[M3]{M_COE}
S. Matsumoto,
General moments of matrix elements from circular orthogonal ensembles,
arXiv:1109.2409v1, 19 pp.

\bibitem[MN]{MN}
S. Matsumoto and J. Novak,
Jucys-Murphy elements and unitary matrix integrals,
Int. Math. Res. Not. (2012), 36 pages.
DOI 10.1093/imrn/rnr267.

\bibitem[S]{speicher}
R. Speicher,
Combinatorial theory of the free product with amalgamation and operator-valued free probability theory.
Mem. Amer. Math. Soc. 132, 627 (1998), x+88.

\bibitem[We]{We78}
D. Weingarten,
Asymptotic behavior of group integrals in the limit of infinite rank.
 J. Mathematical Phys. {\bf 19} (1978), no. 5, 999--1001.

\bibitem[Wi]{Wishart}
J. Wishart,
The generalised product moment distribution in samples from a normal multivariate population.
Biometrika 20A, 1/2(jul 1928), 32-52.



\end{thebibliography}
\end{document}